\documentclass[11pt]{amsart}
\usepackage[usenames,dvipsnames]{pstricks}
\usepackage[mathscr]{eucal}
\usepackage{amsfonts,amsmath,amssymb,amsthm,amscd
	,amsxtra,stmaryrd}
\usepackage[all,2cell,ps]{xy}
\usepackage[notcite, notref]{}
\usepackage[pagebackref]{hyperref}
\usepackage{calc,color}
\usepackage{wasysym}
\usepackage{mathptmx}
\theoremstyle{plain} 

\newtheorem{thm}{Theorem}[section]
\newtheorem{dfn}[thm]{Definition}

\newtheorem{eg}[thm]{Example}

\newtheorem{rmk}[thm]{Remark}
\newtheorem{prop}[thm]{Proposition}

\newtheorem{cor}[thm]{Corollary}
\newtheorem{lem}[thm]{Lemma}

\numberwithin{equation}{section}
\newcommand{\Ann}{\mbox{Ann}\,}

\newcommand{\Hom}{\mbox{Hom}\,}
\newcommand{\Ext}{\mbox{Ext}\,}

\newcommand{\Spec}{\mbox{Spec}\,}

\newcommand{\Ass}{\mbox{Ass}\,}

\newcommand{\Soc}{\mbox{Soc}\,}

\newcommand{\gr}{\mbox{grade}\,}

\newcommand{\depth}{\mbox{depth}\,}
\renewcommand{\dim}{\mbox{dim}\,}

\newcommand{\id}{\mbox{inj.dim}\,}
\newcommand{\vd}{\mbox{vdim}\,}
\newcommand{\pd}{\mbox{proj.dim}\,}

\newcommand{\gd}{\mbox{G-dim}\,}

\newcommand{\h}{\mbox{ht}\,}
\newcommand{\im}{\mbox{Im}\,}

\newcommand{\ag}{\mathcal{G}}

\newcommand{\G}{\mbox{G}}

\newcommand{\fm}{\mathfrak{m}}
\newcommand{\fp}{\mathfrak{p}}
\newcommand{\fq}{\mathfrak{q}}

\begin{document}
	\title{Some characterizations of special rings by  delta-invariant}
	
	\author[M. T. Dibaei]{Mohammad T. Dibaei$^1$} 
	\author[Y. Khalatpour]{Yaser Khalatpour$^2$}
	
	\address{$^{1, 2}$ Faculty of Mathematical Sciences and Computer,
		Kharazmi University, Tehran, Iran.}
	
	\address{$^{1}$ School of Mathematics, Institute for Research in Fundamental Sciences (IPM), P.O. Box: 19395-5746, Tehran, Iran }
	\email{dibaeimt@ipm.ir} \email{yaserkhalatpour@gmail.com}

	\keywords{linkage of ideals, $\delta$-invariant, generically Gorenstein and Gorenstein rings, regular rings.}
	\subjclass[2010]{13C14, 13D02, 13H10, 16G50, 16E65, 13H05}

	\maketitle
	\begin{abstract}   This paper is devoted to present some characterizations for a local ring to be generically Gorenstein and Gorenstein  by means of $\delta$-invariant and linkage theory.\\
	\end{abstract}

	\section{Introduction}
 Characterizations of important classes of commutative Noetherian rings are in the main interests of the authors. There are a vast number of papers characterizing rings to be regular,  Gorenstein, generically Gorenstein, Cohen-Macaulay, etc (see \cite[Theorem 4.3]{CB}, \cite[Theorem  2.3]{JX},  \cite[Theorem  1.3]{CC},   \cite[Theorem 2.3]{LL}, \cite[Theorem 1]{SD}). For a finite module $M$ over a local ring $R$, the delta invariant of $M$, $\delta_R(M)$, has been defined by M. Auslander \cite{AB}. In this paper we use the delta invariant to study special rings such as generically Gorenstein, Gorenstein and regular.
 
   A commutative Noetherian ring $R$ is called \emph{generically Gorenstein} whenever $R_\fp$ is Gorenstein for all minimal prime ideal $\fp$ of $R$. It is well known that if $(R, \fm)$ is Cohen-Macaulay local ring with canonical module then $R$ is generically Gorenstein if and only if the canonical module is isomorphic to an ideal of $R$ (see \cite[Proposion 3.3.18]{BH}).

		In section 2, we investigate the delta invariant in order to characterize rings to be generically Gorenstein, Gorenstein, or regular. 
Our first result characterizes Cohen-Macaulay local rings with canonical modules which are non-Gorenstein generically Gorenstein (Theorem \ref{g}).	We prove that a complete local ring $(R, \fm, k)$ is regular if and only if $R$ is Gorenstein and a syzygy module of $k$ has a principal direct summand $R$-module whose delta invariant is equal $1$ and satisfies an extra condition (see Theorem \ref{2}).

		Section 3 is devoted to characterize Gorenstein  rings to be non-regular by means of higher delta invariant. We end the section with the following result.
		 
		 {\bf Corollary 
		 	 \ref{cor}}. {\it Suppose that $(R, \fm)$ is a  Gorenstein local ring  of dimension $d$ such that $R/\fm $ is infinite. Consider the following statements.
		\begin{itemize}
			\item [(a)]  $R$ is not regular.
			\item [(b)] There exists an  $ \fm $-primary ideal $I$ of $R$ such that
			\begin{itemize}
				
				\item[(i)]  $ I^{i}/I^{i+1} $ is free $R/I $-module for any $ i\geq0 $, and
				\item[(ii)]  for all $R$-regular sequence ${\bf x}=x_1,\cdots, x_s$ in $ I $ with \\ $ x_i+(x_1,\cdots, x_{i-1}) \in ( I/(x_1,\cdots, x_{i-1}))\setminus (I/(x_1,\cdots, x_{i-1}))^{2}$, $1\leq i\leq s$,   $\delta^{n}_{R/{\bf x}R}\left( R/I\right)=0$ for all $ n\geq0 $. 
			\end{itemize}
			\item [(c)] There exists a non-zero ideal $I$ of $R$ such that
			\begin{itemize}                       
				\item[(i)] $ \emph\depth( G )\geq \emph\depth_R(R/I) $, and
				\item[(ii)] $\delta_R^{n}(R/I^{m})=0 $ for all  integers $  n\geq d-\emph\depth( G )+1 $ and $ m\geq1 $. 
			\end{itemize}
		\end{itemize}
		Then  the implications (a)$\Rightarrow$(b)  and  (b)$\Rightarrow$(c)  hold true. If $ \emph\depth(G)>\emph\depth_R(R/I) $, the  implication  (c)$\Rightarrow$(a) holds true. }
		
			In section 4, we are interested in characterization of a ring to be Gorenstein by means of linkage theory and show how generically Gorenstein-ness of a ring with positive dimension may be reducible to the one with smaller dimension. As a result a class of non-generically Gorenstein rings may be recognized. 
						 
						Here is our main result in section 4. If $I$ and $J$ are ideals of a Cohen-Macaulay local ring with canonical module $\omega_R$ such that  $0:_{\omega_R} I=J\omega_R $ and $0:_{\omega_R} J=I\omega_R $ (these conditions coincide with those of  Peskine-Szpiro's conditions when the base ring is Gorenstein, see \cite{PS}), then Cohen-Macaulay-ness of $R/I$ and of $R/J$ are equivalent provided the $\G$-dimension of some particular modules are finite (see Theorem \ref{H}). We use this result to show when $\delta_R^i(I\omega_R)$ vanishes for all $i\geq 0$, where $\delta_R^i$ denote the higher delta invariants.
	
		Throughout $( R,\fm)$ is a commutative local Noetherian ring with maximal ideal $\fm $ and residue field $k=R/\fm $, and all modules are finite (i.e. finitely generated). 
	\section{Characterization  by delta invariant}
	The delta invariant of $M $ has been introduced by M. Auslander in  paragraph just after  \cite[Proposition 5.3]{AB}. For a finite (i.e. finitely generated) $R$-module $M$, denote $M ^\text{cm}$ the sum of all submodules $\phi(L)$ of $M $, where $L $ ranges over all maximal Cohen-Macaulay $R$-modules with no non-zero free direct summands and $\phi $ ranges over all $R$-linear homomorphisms from $L$ to $M $. The  $\delta$ invariant of $M$, denoted by $\delta_{R}\left( M\right)$,  is defined to be $\mu_R(M/M^\text{cm})$, the minimal number of generators of the  the quotient module $M/M^\text{cm} $. 
	
	A short exact sequence $0\longrightarrow Y \longrightarrow X \overset{\varphi}{\longrightarrow} M \longrightarrow 0 $ of $R$-modules is called a Cohen-Macaulay approximation of $ M $ if $X$ is maximal-Cohen-Macaulay $R$-module and $Y$ has finite injective dimension over $R$. A Cohen-Macaulay approximation  $0\longrightarrow Y \longrightarrow X \overset{\varphi}{\longrightarrow} M \longrightarrow 0 $ of $M$ is called minimal if each endomorphism $ \psi $ of $ X $, with $ \varphi\circ \psi=\varphi $, is an automorphism of $ X $. If $R$ be Cohen-Macaulay with canonical module $ \omega_R$ then  a minimal Cohen-Macaulay approximation of $ M $ exists and is unique up to isomorphism (see \cite[Theorem 11.16]{GL},  \cite[Corollary 2.4]{MA} and \cite[Proposition 1.7]{MR}). If the sequence $ 0\longrightarrow Y \longrightarrow X \overset{\varphi}{\longrightarrow} M \longrightarrow 0 $ is a minimal Cohen-Macaulay approximation of $ M $, then  $ \delta_R(M) $ determines  the maximal rank of a free direct summand of $ X $ (see \cite[Exercise 11.47]{GL} and \cite[Proposition 1.3]{MM}). Also it can be shown that if $R$ is Cohen-Macaulay ring which admits a canonical module, then $\delta_R(M) $ is less than or equal to $ n $, where  there is an epimorphism $X\oplus R^{n}\longrightarrow M$ with $X $ a maximal Cohen-Macaulay module with no free direct summands (see \cite[Proposition 11.25]{GL} and  \cite[Proposition 4.8]{AA}). This definition of delta is used by Ding \cite{BB}.  We recall the basic properties of the delta invariant. 
\begin{prop}\label{a} 
\cite[Corollary 11.26]{GL} and  \cite[Lemma 1.2]{AD}. Let $M $ and $N $ be finite modules over a Gorenstein local ring $(R, \fm, k)$. Then the following statements hold true.	
\begin{itemize} 
\item [(i)] $\delta_{R}(M\oplus N)=\delta_{R}\left( M\right)+\delta_{R}\left( N\right)  $.
\item[(ii)] If  there is a an $R$-epimorphism $M\longrightarrow N $, then $\delta_{R}\left( M\right)\geq \delta_{R}\left( N\right)$.
\item[(iii)] $\delta_{R}(M)\leq\ \mu(M) $.
\item[(iv)] $\delta_{R}\left( k\right)=1$  if and only if $R$ is regular.
\item[(v)] $\delta_{R}(M)=\mu(M) $ when  $\emph\pd_{R}(M) $ is finite.	
\end{itemize}
\end{prop}
	
 Here is our first observation which shows how one may characterize a Cohen-Macaulay local ring with canonical module to be generically Gorenstein by the $\delta$-invariant.

  \begin{thm}\label{g} Let $(R, \fm)$ be a    Cohen-Macaulay local ring of dimension $d>0$  with canonical module   $\omega_{R}$. Then the following statements are equivalent.
  	\begin{itemize}
  		\item [(a)] The ring $R$ is generically Gorenstein and non-Gorenstein.
  		\item [(b)] There exists an ideal $I$ of $R$ such that
  		\begin{itemize}
  			\item[(i)] $\delta_{R}\left( R/I\right)=1$,    
  			\item[(ii)] $\emph\h_{R}\left( I\right) = 1$, 
  			\item[(iii)]  There exists a commutative diagram 
  			
  			\begin{center}
  				$\begin{array}{cccccccccc}
  				& & R & \longrightarrow & R/I & \longrightarrow & 0\\ 
  				& &\downarrow\cong & & \downarrow\cong &  &  \\ 
  				&  &\emph\Hom_{R}(I,\omega_{R})  & \longrightarrow &\emph\Ext^{1}_{R}(R/I,\omega_{R})  &\longrightarrow & 0
  				\end{array}$ 
  			\end{center}
  			
  			with isomorphism vertical maps.
  		\end{itemize}
  	\end{itemize}
  \end{thm}
  \begin{proof}
  	(a)$\Rightarrow$(b). Assume that $R$ is generically Gorenstein and that $\omega_{R}\ncong R$.   As $\omega_{R}$ is an ideal of $R$, we consider the exact sequence
  	\begin{center}
  		$0\longrightarrow  \omega_{R} \longrightarrow R\overset{\pi}{\longrightarrow} M \longrightarrow 0,  $
  	\end{center} where $M:=R/\omega_R$.
  	Let $L$ be a maximal Cohen-Macaulay $R$-module with no free direct summands,  $\phi:L\longrightarrow M$ an $R$-homomorphism. Applying the functor $\Hom_R(L,-)$
  	gives the long exact sequence
  	\begin{center}
  		$0\longrightarrow \Hom_{R} (L, \omega_{R} ) \longrightarrow \Hom_{R} (L,R) \longrightarrow \Hom_{R} (L,M )\longrightarrow \Ext^{1}_{R}(L,\omega_{R}).$
  	\end{center}
  	As $\Ext^{1}_{R}(L,\omega_{R})=0$, there exists $\alpha\in \Hom_{R} (L,R)$ such that $\pi \circ \alpha =\phi $.  If there exists $x\in L$ such that $\phi(x)\notin \fm M$ then  we have $\alpha(x)\notin \fm$, i.e. $\alpha(x)$ is a unit and so $\alpha$ is an epimorphism which means $L$ has a free direct summand which is not the case.  Hence $\phi(L)\subseteq \fm M $. Therefore $M^\text{cm} \subseteq \fm M$ and we have
  	$$\delta_{R}(M)=\mu(M/M^\text{cm})=	\vd_{k}(M/(M^\text{cm}+\fm M))=\mu(M/\fm M)=\mu(M)=1. $$
  	Moreover, we have  $\Ext_{R}^{1}\left( R/\omega_{R},\omega_{R}\right) \cong R/\omega_{R}$  since $R/\omega_{R}$  is Gorenstein ring of dimension $d-1$, and $\Hom_{R} (\omega_{R} ,\omega_{R}
  	)\cong R$, $\h_{R}\left( \omega_{R}\right) =1$. Now that the statement (iii) follows naturally.
  	
  	(b)$\Rightarrow$(a). As $\h(I)= 1$, $I\nsubseteq\underset{\fp\in\Ass(R)}{\cup}\fp$ and so $\Hom_{R}(R/I,\omega_{R})=0$. Hence, naturally, we obtain the exact sequence
  	
  	\begin{center}
  		$0\longrightarrow  \Hom_{R}(R,\omega_{R})\longrightarrow \Hom_{R}(I,\omega_{R})\longrightarrow \Ext^{1}_{R}(R/I,\omega_{R})\longrightarrow 0.$
  	\end{center}
  	One has the following commutative diagram
  	
  	\begin{center}
  		$\begin{array}{cccccccccc}
  		0 & \longrightarrow  & I & \longrightarrow & R & \longrightarrow & R/I &\longrightarrow &0& \\ 
  		&  & &  & \downarrow\cong &  & \downarrow \cong&  &  \\ 
  		0 & \longrightarrow & \Hom_{R}(R,\omega_{R}) & \longrightarrow &\Hom_{R}(I,\omega_{R})  & \longrightarrow &\Ext^{1}_{R}(R/I,\omega_{R})  &\longrightarrow  & 0.
  		\end{array}$
  	\end{center}
  	Therefore we obtain,  $I\cong \omega_{R} $ which means $R$ is generically Gorenstein.  
  	
  	To see the final claim, assume contrarily that $R$ is Gorenstein. Hence  $ \omega_R\cong R$ and $  \Hom_{R}(R,\omega_{R})\cong \Hom_{R}(I,\omega_{R})  $.  Now, the commutative diagram (iii) implies that $R/I=0 $ so $ \delta_R(R/I)=0 $ which is a contradiction. 
  \end{proof}

Our next observation traces ideals $J$ with $\delta_R(R/J)=1$.

 \begin{prop} 
	Let $(R, \fm)$ be a Cohen-Macaulay local ring with canonical module   $\omega_R$, and let $I$ and $J$ be two ideals of $R$ such that $R/J\cong\Omega_R(R/I)$, $J\cong \emph\Ann_{\omega_R}(I) $ and $ \emph\h_R(J)\geq 2$. Then $R$ is generically Gorenstein and $ \delta_R(R/J)=1 $.
\end{prop}
\begin{proof}
	Apply the functor $\Hom_R(-, \omega_R) $ on the exact sequence $ 0\longrightarrow R/I\longrightarrow \overset{t}{\oplus}R\longrightarrow R/J\longrightarrow 0 $ to get the exact sequence
	\begin{center}
		$ 0\longrightarrow \Hom_R(R/J, \omega_R)\longrightarrow \Hom_R(\overset{t}{\oplus}R, \omega_R)\longrightarrow \Hom_R(R/I, \omega_R)\longrightarrow \Ext^{1}_R(R/J, \omega_R)\longrightarrow0 $.
	\end{center}
	But we have $ \Hom_R(R/J, \omega_R)=0=\Ext^{1}_R(R/J, \omega_R)=0 $, since $ \h_R(J)\geq 2  $.  Therefore we have $\overset{t}{\oplus}\omega_R\cong J$. Then $R$ is generically Gorenstein.
	
	Finally, assume that $L$ is a maximal Cohen-Macaulay $R$-module with no free direct summands and that $f: L\rightarrow R/J$ is an $R$-homomorphism. Applying $\Hom(L, -)$ on the exact sequence $$ 0\longrightarrow \overset{t}{\oplus}\omega_R\cong J\longrightarrow R \longrightarrow R/J \longrightarrow 0 $$ implies that $f(L)\subseteq\fm(R/J)$, so that $\delta_R(R/J)=1$.
\end{proof}
  Over a Gorenstein local ring $R$,  Proposition \ref{a} (iii) states that the inequality   $\delta_R(M)\leq\mu(M)$. In the following, we explore when equality holds true by means of Gorenstein dimensions. 
     
  A finite $R$-module $ M $ is said to be {\it totally reflexive} if the natural map $ M\longrightarrow \Hom_{R}(\Hom_{R}(M, R), R) $ is an isomorphism and
  $ \Ext^{i}_{R}(M, R)=0=\Ext^{i}_{R}(\Hom_{R}(M, R), R)$ for all $i>0$.
  An $R$-module $ M $  is said to have  Gorenstein dimension $\leq n$, write $\gd_R(M)\leq n $, if there exists an exact sequence  $$ 0\longrightarrow G_{n}\longrightarrow\cdots \longrightarrow G_{1}\longrightarrow G_{0}\longrightarrow M\longrightarrow 0, $$ of $R$ modules such that each $ G_{i} $ is totally reflexive. Write $ \gd_{R}(M)= n $ if there is no such sequence with shorter length. If there is no such finite length exact sequence, we write $\gd_R(M)=\infty$.

  Our first result indicates the existence of a finite length $R$- module $M$ such that the equality $\delta_R(M)=\mu(M)$ holds true may put a strong condition on $R$. More precisely: 
    
     In \cite[Theorem 6.5]{CB}, it is shown that the local ring $(R, \fm, k)$ is Gorenstein if and only if  $\Omega^{n}_{R}(k) $  has a G-projective summand for some $n$, $ 0\leq n\leq \depth R+2 $ .
 \begin{thm} \label{l}
 Let $(R, \fm) $ be a local ring. The following statements are equivalent.  
\begin{itemize}
\item[(i)] $R$ is Gorenstein.       
 \item[(ii)] There exists an $R$--module $M$ such that $ \delta_R(M)=\mu(M)$, $ \fm^{n}M=0$, and $\emph{\gd}_{R}(\fm^{n-2}M^{\emph{cm}})<\infty$ for some integer $ n\geq 2 $.
 \end{itemize}
 \end{thm}
 \begin{proof}
Assume first that $R$ is Gorenstein and that $\underline{x}$ is a maximal $R$--regular sequence. Thus there is a surjective homomorphism $R/\fm^{t}\longrightarrow R/\underline{x}R$ for some  integer $t\geq 1$. As $ \pd(R/\underline{x}R)<\infty $,  Proposition \ref{a} implies that
\begin{center}
 $1=\mu(R/\underline{x}R)=\delta_R(R/\underline{x}R)\leq\delta_R(R/\fm^{t})\leq\mu(R/\fm^{t})=1.$
 \end{center}
 Therefore $\delta_R(R/\fm^{t})=1=\mu(R/\fm^{t})$. Now by setting  $ n=t+1\geq2 $, the module $M:=R/\fm^{t}$ trivially justifies claim (ii).
       
 For the converse, consider the natural exact sequence  
 \begin{center}
 $0\longrightarrow (M^{\text{cm}}+\fm M)/\fm M\longrightarrow M/\fm M\longrightarrow \frac{M/M^{\text{cm}}}{\fm(M/M^{\text{cm}})}\longrightarrow 0 $.
  \end{center} 
 Now the equality $\delta_R(M)=\mu(M)$ implies that $M^{\text{cm}}\subseteq \fm M $. As $\fm^nM=0$, $ \fm^{n-2}M^{\text{cm}} $ is vector space. Our assumption $\gd_{R}(\fm^{n-2}M^\emph {cm})<\infty$ implies that $\gd_{R}(R/\fm)<\infty$.  Hence  $R$ is Gorenstein by \cite[Theorem 1.4.9]{BA}.
 \end{proof}
  

      Assume that $ (R, \fm, k) $ is a local ring with residue field $k$. In \cite[Corollary 1.3]{CC}, Dutta presents a characterization for $R$ to be regular in terms of the admitting a syzygy of $k$ with a free direct summand. Later on, 
  Takahashi, in \cite[Theorem 4.3]{CB}, generalized  the result in terms of the existence of a syzygy module of the residue field having a semidualizing module as its direct summand. Also Ghosh et.al, in \cite[Theorem 3.7]{BD}, have shown that the ring is regular if and only if  a syzygy module of $ k $  has a non-zero direct summand of finite injective dimension.  
 
 Now we investigate these notions by means of delta invariant. Denote by $\Omega^{i}_{R}(k) $ the  $i$th syzygy, in the minimal free resolution, of $k$. In the next result we prove that, for $i\gg 0$, $\Omega_R^i(k)$ does not possess a direct summand of finite injective dimension by means of $\delta$-invariant. D. Ghosh informed us that this result has been proved by him directly \cite[Theorem 3.3]{GO}.  

  \begin{prop} \label{f}
 Let $(R, \fm, k)$ be a local ring of dimension $d$. Then, for any $ i>d $, the following statements hold true.
 \begin{itemize}
 	\item [(a)]  $\Omega^{i}_{R}(k) $ has no non-zero direct summand of finite injective dimension.
   \item[(b)] If $R$ is Cohen-Macaulay  with canonical module $ \omega_R$, then $\Omega^{i}_{R}(k) $ has no direct summand isomorphic to $ \omega_R$.
  \item[(c)] $\Omega^{i}_{R}(k) $ has no free direct summand. 
 \end{itemize}
 \end{prop} 
 \begin{proof}
(a).  Assume contrarily that, for an integer $i>d$,  $\Omega^{i}_{R}(k)=M\oplus N$ for some $R$--modules $ M\neq 0 $ and $ N $ with $ \id_R(M)<\infty$. By \cite[Theorem 3.7]{BD} we have $R$ is regular so that $ \pd_R(M)<\infty$. Thus, by Proposition \ref{a}(v), $ \delta_R(M)=\mu(M)$. As $i-1\geq d$,  $\Omega^{i-1}_{R}(k) $ is a maximal Cohen-Macaulay $R$-module. By the paragraph just after  \cite[Proposition 5.3]{AB}, the exact sequence $0\longrightarrow \Omega^{i}_{R}(k)\longrightarrow \oplus R\longrightarrow \Omega^{i-1}_{R}(k)\longrightarrow 0$ implies that $ \delta_R(\Omega^{i}_{R}(k))=0 $.   Therefore by Proposition \ref{a}(i) we obtain $ \mu(M)=0 $ which contradicts the fact that $ M\neq 0$.

Parts (b) and (c) are clear conclusions of (a).
\end{proof}
 
\begin{dfn}
 An $R$-module $ X $ is said to satisfy the condition $ (\ast) $ whenever, for any  $ X $-regular element  $ a $, $ X/aX $ is indecomposable as $R/aR$-module.
\end{dfn}

 \begin{thm} \label{2}
 Let $( R, \fm, k) $ be a complete local ring of dimension $ d $. The following statements are equivalent.  
 \begin{itemize}
 \item[(i)] $R$ is regular.     
 \item[(ii)] $R$ is Gorenstein and $\Omega^{n}_{R}(k) $ has a principal  $R$-module as its direct summand whose delta invariant is $1$ and satisfies the property $ (\ast) $,  for some $ n\geq 0$.
 \end{itemize}
 \end{thm}
 \begin{proof}
 (i)$\Rightarrow$(ii). $k=\Omega^{0}_{R}(k) $ fulfills our statement by Proposition \ref{a}.
 
 (ii)$\Rightarrow$(i). Suppose that $R$ is Gorenstein and, for an integer $n\geq 0$, $\Omega^{n}_{R}(k)\cong X\oplus Y $ for some $R$-modules $ X $ and $ Y $ such that $X\cong R/\Ann_R(X)$  with $ \delta_R(X)=1 $. The case $ n=0 $ implies that $R$ is regular. So we may assume that $ n\geq 1$. 
 
 We proceed by induction on $d$.
 For the case $ d=0 $, if $ \fm\neq0 $ then $ \Soc(R)\neq0 $ and $R/\Soc(R) $ is maximal Cohen-Macaulay $R$-module with no free direct summand and so $ \delta_R(R/\Soc(R))=0 $. On the other hand, by  \cite[Lemma 2.1]{BD}, $ \Soc(R)\subseteq \Ann_R(\Omega^{n}_{R}(k) )=\Ann_R(X\oplus Y)\subseteq \Ann_R(X)$.  Therefore the natural surjection $R/\Soc(R)\longrightarrow R/\Ann_R(X)\cong X $ implies that $ 1=\delta_R(X)\leq \delta_R(R/\Soc(R))=0$ which is absurd. Hence $ \fm=0 $ and $R=R/\fm $ is regular. 
 
 Now we suppose that $ d\geq1 $ and the statement is settled for $ d-1 $. As $R$ is Cohen-Macaulay, we choose an $R$-regular element $ y\in \fm\setminus\fm^{2} $. Hence  $ y $ is $\Omega^{n}_{R}(k)  $-regular and $X$-regular.  We set $ \overline{(-)}=(-)\otimes_R R/yR$. Note that $ \overline{X} $ is a principal  $ \overline{R}$-module and that, by \cite[Corollary 2.5]{CA} and Proposition \ref{a}, $ 1=\delta_R(X)\leq \delta_{\overline{R}}(\overline{X}) \leq \mu(\overline{X})=1 $. Note that, by \cite[Corollary 3.5]{CB} ,  we have
  \begin{center}
 $ \overline{\Omega^{n}_{R}(k)}\cong \Omega^{n}_{\overline{R}}(k)\oplus \Omega^{n-1}_{\overline{R}}(k).  $
 \end{center}
 Therefore we have $ \overline{X}\oplus \overline{Y} \cong\overline{\Omega^{n}_{R}(k)}\cong \Omega^{n}_{\overline{R}}(k)\oplus \Omega^{n-1}_{\overline{R}}(k)  $. But $ \overline{X} $ is indecomposable $ \overline{R} $-module so,  by Krull-Schmit uniqueness theorem (see  \cite[Theorem 21.35]{BN}), $ \overline{X} $ is direct summand of $ \Omega^{n-1}_{\overline{R}}(k) $ or $ \Omega^{n}_{\overline{R}}(k) $. Now our induction hypothesis implies that $ \overline{R} $ is regular and so is $R$.
 \end{proof}  We end this section by the following remark which gives some more informations on the delta invariant.
  \begin{rmk} \label{n}
 Let $(R, \fm)$ be a local ring. 
 \begin{itemize}
 	\item[(a)] The ring $R$ is regular if and only if $R$ is  Gorenstein and  $\delta_R(M)>0$  for all non-zero finitely generated  $R$-module $ M $. 
 	\item[(b)]  If $R$ is Cohen-Macaulay with canonical module $ \omega_R$, then $R$ is not Gorenstein if and only if there exists a non-zero $R$-module $M$ with  $ \delta_R(M)=0$  and $\emph\id_{R}(M)<\infty $. 
 \end{itemize}
 \end{rmk}
\begin{proof}
 (a). Suppose that $R$ is regular. Assume contrarily that there exists a non-zero $R$--module $ M $ such that $\delta_R(M)=0$.  By definition of delta, there exists  a surjective homomorphism $ X\longrightarrow M $ such that $ X $ is maximal Cohen-Macaulay $R$-module  with no free direct summand. On the other hand, as $R$ is regular,   $ \pd_R(X)=0 $ and so  $ X $ is free a $R$-module which  is not the case.
 
 Conversely, by assumption and Proposition \ref{a},  $ 1\leq\delta_R(R/\fm)\leq \mu(R/\fm)=1$. Hence,  by Proposition \ref{a}, $R$ is regular.
  
 (b). Assume  that $R$ is not Gorenstein.  As $R$ is   indecomposable then $\omega_R$ is maximal Cohen-Macaulay with no free direct summand and so  we trivially have   $ \delta_R(\omega_{R})=0 $ with $ \id_R(\omega_{R})<\infty $.
 
 Conversely, assume contrarily that  $R$ is Gorenstein. As $\id_R(M)<\infty$, we have $ \pd_R(M)<\infty $,  by  \cite[Exercise 3.1.25]{BH}.   Hence Proposition \ref{a} implies that $ 0=\delta_R(M)=\mu(M)$ which contradicts the fact that $ M\neq 0$.
 \end{proof}

 \section{Gorenstein non-regular rings}

   For an integer $n\geq 0$ and an $R$--module $M$, we denote $\delta_{R}^{n}(M):=\delta_R(\Omega^{n}_{R}(M)) $ as the higher delta invariant, where $ \Omega^{n}_{R}(M) $ is $ n $th  syzygy module of $ M $ in its minimal free resolution. The following result will be used in characterizing a ring to be non-regular Gorenstein of dimensions $\leq 2$.

   \begin{prop}\label{sa}
  Assume that $(R, \fm)$ is a $1$-dimensional  Gorenstein  local ring and that $ I $ is an $ \fm $-primary ideal of $R$ such that $ \mu_R(I)\geq2 $. Then 
  $\delta_R(I^{n}/I^{n+1})=0 $ and $ \delta_R^{m}(R/I^{n})=0 $ for all positive integers $ n $ and $ m $. 
    \end{prop}
   \begin{proof}
    The assumption  $ \dim(R/I)=0 $ and the exact sequence $ 0\longrightarrow I\longrightarrow R\longrightarrow R/I\longrightarrow0 $ imply that $ I $ is a maximal Cohen-Macaulay as an $R$--module.  As $ \mu_R(I)\geq2 $,  $ I $ has no free direct summand and so $ \delta_R(I)=0 $.  For a finite $R$-module $M$, the natural epimorphism  $ \overset{\mu_R(M)}{\oplus}I \longrightarrow IM $, by Proposition \ref{a}, implies that $ \delta_R(IM)\leq \delta_R(\overset{t}{\oplus}I)=0 $ which gives $ \delta_R(IM)=0 $. Let $ n $ and $ m $ be positive integers. As $ \Omega^{m}_{R}(R/I^{n}) $ is a maximal Cohen-Macaulay $R$-module for all $ m\geq1 $, we have $ \delta_R^{m}(R/I^{n})=0 $ for all $ m>1 $ (see the paragraph just after  \cite[Proposition 5.3]{AB}). For the case $ m=1 $ we have   $\delta_R^{1}(R/I^{n})=\delta_R(\Omega^{1}_{R}(R/I^{n}))=\delta_R(I^{n})=\delta_R(II^{n-1})=0 $.
    \end{proof}
      
       For an $R$-regular element $ x$ in $\fm $, we set $ \overline{(-)}=(-)\otimes_R R/xR$.   Recall from the first paragraph of section 5 of \cite{AB} that an $R$-module $ M $ is called {\it weakly liftable} on $ \overline{R}$ if $ M $ is a direct summand of $\overline{N} $ for some $R$-module $N$.  We recall the following result from  \cite[Proposition 5.7]{AB} and provide a proof by using of \ref{sa} for convenience of the reader.
           \begin{rmk} \label{12}
                     Suppose that $(R, \fm, k)$ is a  Gorenstein  local ring of dimension $ d\geq1 $. If $R$ is not regular then $ \delta_R^{n}(R/\fm)=0 $ for all  integer $ n\geq0 $.
                     \end{rmk}
                     \begin{proof}
                         Suppose that $R$ is not regular. If $ n=0 $ then we have nothing to prove. For $ n\geq1 $, we prove by induction on $ d $. If $ d=1 $  then  $ \mu_R(\fm)\geq2$  and result follows by proposition \ref*{sa}. Now assume that $ d \geq 2 $ and that the  result has been proved for $< d $. Choose an $R$-regular element $ x\in \fm \setminus \fm^{2} $, set $ \overline{R}=R/xR$ and $ \overline{\fm}=\fm/xR$. We have,  by \cite[Corollary 3.5]{CB}, $ \overline{\Omega^{n}_{R}(k)}\cong \Omega^{n}_{\overline{R}}(k)\oplus \Omega^{n-1}_{\overline{R}}(k) $ and so
                          \[\begin{array}{rl} 
                               \delta_R^{n}(k)=\delta_R(\Omega^{n}_{R}(k))&\leq\delta_{\overline{R}}(\overline{\Omega^{n}_{R}(k)})\\
                                &=\delta_{\overline{R}}(\Omega^{n}_{\overline{R}}(k))+\delta_{\overline{R}}(\Omega^{n-1}_{\overline{R}}(k))\\
                               \text{(by induction hypothesis)} &=0+0\\
                                &=0.\\
                                \end{array}\]
                        
                     \end{proof} 
    \begin{cor} \label{7}
   	Assume that $(R, \fm)$ is a Gorenstein  local ring of positive dimension $d\leq2$. Then the following statements are equivalent.
   	\begin{itemize}
   		\item [(a)]  $R$ is not regular.
   		\item [(b)] There exists an  $ \fm $-primary ideal $I$ of $R$ such that
   		
   		\begin{itemize}
   			\item[(i)]  $ \mu_R(I)\geq d+1 $,
   			\item[(ii)] $ I^{d-1}/I^{d} $ is a free $R/I $-module, 
   			\item[(iii)]  There exists $R$-regular element $ x\in I^{d-1} \setminus I^{d} $ such that the natural map $ R/I\overset{x.}{\longrightarrow} I^{d-1}/I^{d} $ is injective and $\delta_{R/xR}\left( R/I\right)=0$.
   		\end{itemize}
   		\item [(c)] There exists a non-zero ideal $I$ of $R$ such that $ \delta_R^{n}(R/I)=0 $ for all positive integers $  n\geq1 $.
   	\end{itemize}
   \end{cor}
   \begin{proof}
   	(a)$\Rightarrow$(b).  We set $ I=\fm $. If $ d=1 $ the result is trivial. For $ d=2 $, choose an $R$-regular element $ x\in \fm \setminus \fm^{2} $. Hence the map $ R/\fm\longrightarrow \fm/\fm^{2} $, with $a+\fm\rightsquigarrow ax+\fm^2$, is injective. As $ R/xR $ is not  regular,  $\delta_{R/xR}\left( R/\fm\right)=\delta_{R/xR}\left( (R/xR)/(\fm/xR) \right)=0 $ by Remark \ref{12} and also it is clear $ \fm/{\fm}^{2} $ is free $ R/\fm $ module.
   	
   	(b)$\Rightarrow$(c).   Set $ \overline{R}=R/xR$ and assume that $ n\geq1 $. If $ d=1 $, the result follows  by Proposition \ref{sa}. Let $ d=2 $. As $ \dim(\overline{R})=1 $ and $ \mu_{\bar R}(\bar I)\geq 2 $, by Proposition \ref{sa} and assumption (iii), we have  $ \delta_{\overline{R}}(\Omega^{n}_{\overline{R}}(R/I))=\delta_{\overline{R}}(\Omega^{n}_{\overline{R}}(\overline{R}/\overline{I}))=0$ and $\delta_{\overline{R}}(\Omega^{n-1}_{\overline{R}}(R/I))=\delta_{\overline{R}}(\Omega^{n-1}_{\overline{R}}(\overline{R}/\overline{I}))=0$. On the other hand, $ \dim R/I=0 $ and the exact sequence $$ 0 \longrightarrow R/I\longrightarrow I/I^{2}\longrightarrow I/(xR+I^{2})\longrightarrow 0 $$  imply that $ \pd_{R/I}(I/(xR+I^{2})) =0 $ by Auslander-Buchsbaum formula.  Therefore, the commutative diagram 
   	\begin{center}
   		$\begin{array}{cccccccccc}
   		0 & \longrightarrow  & R/I & \overset{x.}{\longrightarrow} & I/xI & \longrightarrow &I/xR &\longrightarrow &0& \\ 
   		& & \parallel &  & \downarrow  &  & \downarrow\\ 
   		0 & \longrightarrow & R/I & \overset{x.}{\longrightarrow} &I/I^{2}  & \longrightarrow &I/(xR+I^{2})  &\longrightarrow  & 0
   		\end{array}  $
   	\end{center} 
   with exact rows implies that the upper row 
   	splits. Hence $R/I $, as $\bar{R}$--module,  is weakly liftable on $\bar{R}$.   Thus, by \cite[Proposition 5.2]{CB}, we obtain 
   	$ \overline{\Omega^{n}_{R}(R/I)}\cong \Omega^{n}_{\overline{R}}(R/I)\oplus \Omega^{n-1}_{\overline{R}}(R/I) $. 
   	Therefore 
   	 \[\begin{array}{rl} 
   	     \delta_R^{n}(R/I)= \delta_R(\Omega^{n}_{R}(R/I))&\leq\delta_{\overline{R}}(\overline{\Omega^{n}_{R}(R/I)})\\
   	       &=\delta_{\overline{R}}(\Omega^{n}_{\overline{R}}(R/I))+\delta_{\overline{R}}(\Omega^{n-1}_{\overline{R}}(R/I))\\
   	       &=0+0\\
   	       &=0.\\
   	       \end{array}\]
   	
   	(c)$\Rightarrow$(a).   By assumption $ \delta_R(I)=0 $ therefore, by Remark \ref{n}, $R$ is not regular.
   \end{proof}

      It is shown by Yoshino \cite[Theorem 2.3]{y} that, in a non-regular Gorenstein local ring $(R, \fm)$ with  $\depth ( \ag_\fm(R) )\geq d-1$, one has $ \delta_R^{n}(R/\fm^{m})=0 $ for all positive integers $ n $ and $ m $, where $ \ag_\fm(R) $ denote the associate graded ring of $R$ with respect to $\fm$. In order to present a generalization of this theorem, we first bring a lemma. For an ideal $I$ of a ring $R$, we set  $ G:=\ag_I(R) $ as the associated ring of $R$ with respect to $I$.
      
     \begin{lem}\label{sb}
       Assume that $(R, \fm)$ is a local ring  and that $ I $ is an $ \fm $-primary ideal of $R$ such that $ I^{i}/I^{i+1} $ is free $R/I $-module for all $ i\geq0 $ $($ e.g. $ I $ an Ulrich ideal of $R$, see  \cite[Proposition 3.2]{K} $ ) $. Suppose that $ x\in I\setminus I^2 $  such that $ x^{\ast}:=x+I^2 $ is $ G $-regular element  in $ G $.  Set $ \bar{R}=R/xR$. Then, for any $n\geq 0$, we have  $ \Omega^{n}_R(I^{m})\otimes_R \bar{R}\cong  \Omega^{n}_{\bar{R}}(I^{m-1}/I^{m})\oplus \Omega^{n}_{\bar{R}}(I^{m}/xI^{m-1}) $ for all $ m\geq1 $. 
      \end{lem}
      \begin{proof}
     As $ x^{\ast} $ is a $ G $-regular element  in $ G $,  the  map  $ I^{m-1}/I^{m}\overset{x.}{\longrightarrow} I^{m}/I^{m+1}  $ is injective for all $ m\geq1 $. We prove the claim  by induction on $ n $. By \cite[Lemma 2.3]{K} we have $ I^{m}/xI^{m}\cong I^{m-1}/I^{m}\oplus I^{m}/xI^{m-1}$.  
     
      Therefore
       \[\begin{array}{rl} 
     \Omega^{0}_R(I^{m})\otimes_R \bar{R}&=I^{m}\otimes_R \bar{R} \\
      &\cong I^{m}/xI^{m}\\
      &\cong I^{m-1}/I^{m}\oplus I^{m}/xI^{m-1}\\
      &=\Omega^{0}_{\bar{R}}(I^{m-1}/I^{m})\oplus \Omega^{0}_{\bar{R}}(I^{m}/xI^{m-1})
      \end{array},\] 
      which proves the claim for $n=0$.
      
      Now we assume that  $ n> 0 $ and the claim  is settled for integers less than $ n $. Note that a minimal free cover $ 0\longrightarrow  \Omega^{n}_R(I^{m})\longrightarrow F \longrightarrow  \Omega^{n-1}_R(I^{m})\longrightarrow 0$ of $ \Omega^{n-1}_R(I^{m}) $ gives a minimal cover 
      \begin{center}
       $ 0\longrightarrow  \Omega^{n}_R(I^{m})\otimes_R \bar{R}\longrightarrow F\otimes_R \bar{R} \longrightarrow  \Omega^{n-1}_R(I^{m})\otimes_R \bar{R}\longrightarrow 0$ 
       \end{center}
       of $\Omega^{n-1}_R(I^{m})\otimes_R \bar{R}$  over $ \bar{R} $. Hence we get 
       $ \Omega^{n}_R(I^{m})\otimes_R \bar{R}\cong \Omega^{1}_{\bar{R}}(\Omega^{n-1}_R(I^{m})\otimes_R \bar{R})$.
        By  induction hypothesis we have  
        \[\begin{array}{rl}  \Omega^{n}_R(I^{m})\otimes_R \bar{R}&\cong \Omega^{1}_{\bar{R}}(\Omega^{n-1}_R(I^{m})\otimes_R \bar{R})\\
        &\cong \Omega^{1}_{\bar{R}}(\Omega^{n-1}_{\bar{R}}(I^{m-1}/I^{m})\oplus \Omega^{n-1}_{\bar{R}}(I^{m}/xI^{m-1}))\\
        &\cong \Omega^{n}_{\bar{R}}(I^{m-1}/I^{m})\oplus \Omega^{n}_{\bar{R}}(I^{m}/xI^{m-1}).  
    \end{array}\]   
      \end{proof}
       
    \begin{thm} \emph{(Compare \cite[Theorem 2.3]{y})}\label{sf}
    Suppose that $(R, \fm)$ is a  Gorenstein local ring of dimension $d$ with infinite residue field $R/\fm $. Assume that $ I $ is an $ \fm $-primary ideal of $R$ such that:
       	\begin{itemize}
             	\item[(i)]   For any $ i\geq0 $, $ I^{i}/I^{i+1} $ is free $R/I $-module, and
     	\item [(ii)] for any $R$-regular sequence ${\bf x}=x_1,\cdots, x_s$ in $ I $ with  $$ x_i+(x_1,\cdots, x_{i-1}) \in ( I/(x_1,\cdots, x_{i-1}))\setminus (I/(x_1,\cdots, x_{i-1}))^{2}, \ 1\leq i\leq s,$$  we have $\delta^{n}_{R/{\bf x}R}\left( R/I\right)=0$ for all $ n\geq0 $. 
             	\end{itemize} 
          Then  $ \delta_R^{n}(R/I^{m})=0 $ for all  integers $  n\geq d+1-\emph\depth G $ and all $ m\geq1 $. In particular, if $  \emph\depth G  =d-1 $, then  $ \delta_R^{n}(R/I^{m})=0 $ for all $ n \geq 2 $ and all $ m\geq1 $.
    \end{thm}
    \begin{proof}
    Let $ m\geq1 $. If $ d=0 $ the result is trivial. We assume that $ d>0  $ and  $  n\geq d+1-\emph\depth G $.  
    
    If $\depth G=0 $ then $ n\geq d+1 $ and the result is clear by \cite[Corollary 1.2.5]{Av}. Now assume that $ d>0  $ and $ \depth G>0 $. As $ R/\fm $ is infinite, \cite[Lemma 2.1]{HM}  implies that  the map $I^{m-1}/I^{m}\overset{x.}{\longrightarrow} I^{m}/I^{m+1}  $ is injective for some $ x\in I \setminus I^{2} $ with $ x^{\ast}:= x+I^2\in G $ is $ G $-regular. Note that  $ \ag_I(I^{m})\subseteq G $  implies that  $ x\in I $ is non-zero-divisor on $ I^{m} $. Set $ \bar{R}=R/xR $ and $ \bar{I}=I/xR $ and  let $ n\geq d-t+1 $.  By Lemma \ref{sb} we have
       \begin{center}
        $ \Omega^{n-1}_R(I^{m})\otimes_R \bar{R}\cong  \Omega^{n-1}_{\bar{R}}(I^{m-1}/I^{m})\oplus \Omega^{n-1}_{\bar{R}}(I^{m}/xI^{m-1}) $.
        \end{center}
        On the other hand $ x $ is $ \Omega^{n-1}_R(I^{m}) $-regular, therefore by  \cite[Corollary 2.5]{CA} we have
        \begin{center}
        $ \delta_R(\Omega^{n-1}_R(I^{m})\leq\delta_{\bar{R}}(\Omega^{n-1}_R(I^{m})\otimes_R \bar{R}) $.
        \end{center}
       Therefore 
      \begin{equation}\label{4.1}
       \begin{array}{rl} 
       \delta_R^{n}(R/I^{m})&=\delta_R(\Omega^{n-1}_R(I^{m})\\
       &\leq\delta_{\bar{R}}(\Omega^{n-1}_R(I^{m})\otimes_R \bar{R})\\
       &=\delta_{\bar{R}}( \Omega^{n-1}_{\bar{R}}(I^{m-1}/I^{m})\oplus \Omega^{n-1}_{\bar{R}}(I^{m}/xI^{m-1}))\\
       &=\delta_{\bar{R}}(\Omega^{n-1}_{\bar{R}}(I^{m-1}/I^{m}))+\delta_{\bar{R}}(\Omega^{n-1}_{\bar{R}}(I^{m}/xI^{m-1})) \\
       &=\delta_{\bar{R}}^{n-1}(I^{m-1}/I^{m})+\delta_{\bar{R}}^{n-1}(I^{m}/xI^{m-1}).
       \end{array}
       \end{equation}
            
       The  injective  map $I^{m-1}/I^{m}\overset{x.}{\longrightarrow} I^{m}/I^{m+1}  $ implies, by induction on $m$, that $xI^{m-1}=xR\cap I^{m}$ so we get $ (\bar{I})^{m}=I^{m}/(xR\cap I^{m})=I^{m}/xI^{m-1} $, therefore 
        \[\begin{array}{rl}
        \delta_R^{n}(R/I^{m}) &\leq\delta_{\bar{R}}^{n-1}(I^{m}/xI^{m-1})+\delta_{\bar{R}}^{n-1}(I^{m-1}/I^{m})\\&
        =\delta_{\bar{R}}^{n-1}(({\bar{I}})^{m})+\delta_{\bar{R}}^{n-1}(I^{m-1}/I^{m}).           
        \end{array}\]
       Note that, by assumption,  $ I^{m-1}/I^{m}\cong \oplus R/I $, and $ \delta_{\bar{R}}^{n-1}(I^{m-1}/I^{m})=\delta_{\bar{R}}^{n-1}(R/I)=0 $.
        
       If $ d=1 $ then $ \dim(\bar R)=0$  so $\delta_{\bar{R}}^{n-1}(({\bar{I}})^{m})= \delta_{\bar{R}}(\Omega^{n-1}_{\bar{R}}(({\bar{I}})^{m}))=\delta^{n}_{\bar{R}}(\bar{R}/(\bar{I})^{m})=0 $  hence result is clear.
       
        Suppose that $ d\geq2 $.  As $ n\geq d-t+1=(d-1)-(t-1)+1 $, when $t=1$,  we have   $ n\geq (d-1)+1 $ hence $ \delta^{n}_{\bar{R}}(\bar{R}/(\bar{I})^{m})=0 $ , therefore  $\delta_R^{n}(R/I^{m})= 0$. 
        
        If $ t\geq2 $  then $  \depth (  gr_{\bar{R}}(\bar{I}))=\depth (  gr_R(I)/x^{\ast}gr_R(I))=t-1>0  $. As $ \bar R/\bar \fm \cong R/\fm $ is infinite and  $ \dim(\bar R)>0  $ and $  \depth (  gr_{\bar{R}}(\bar{I}))>0  $, by \cite[Lemma 2.1]{HM}, there exists  $ \bar y =y+xR\in \bar I \setminus \bar I^{2} $ such that $ \bar y^{\ast} $ is $ \bar G $-regular. Therefore the map  $ \bar I^{m-1}/\bar I^{m}\overset{\bar y}{\longrightarrow} \bar I^{m}/\bar I^{m+1}  $ is injective. On the other hand we have  $ (\bar I)^{m}/(\bar I)^{m+1}\cong I^{m}/(xI^{m-1}+I^{m+1}) $ and, by  \cite[Lemma 2.3]{K},  $I^{m}/(xI^{m-1}+I^{m+1}) $ is a direct summand of $ I^{m}/I^{m+1} $. Therefore  $ (\bar I)^{m}/(\bar I)^{m+1} $ is a free $ \bar R/\bar I $-module for any $ i\geq1 $.  Also $ gr_{\bar I}(\bar I^{m})\subseteq \bar G $   implies that  $ \bar y\in \bar I $ is non-zero divisor on $ \bar I^{m} $. Set $ \bar {\bar R}=\bar R/\bar y \bar R $ and $ \bar {\bar I}=\bar I/\bar y \bar I $. Then by the same argument as above we have $\delta_{\bar R}^{n}(\bar R/\bar I^{m}) \leq\delta^{n}_{\bar{\bar{R}}}(\bar{\bar{R}}/(\bar{\bar{I}})^{m})+\delta_{\bar{\bar{R}}}^{n-1}(\oplus R/I) $.
        
         By our assumption $ \delta_{\bar{\bar{R}}}^{n-1}(\oplus R/I))= \delta_{R/(x,y)}^{n-1}(\oplus R/I))=0 $. When $ d=2 $, $ \dim(\bar{\bar R})=0$ and so $ \delta^{n}_{\bar{\bar{R}}}(\bar{\bar{R}}/(\bar{\bar{I}})^{m})=0 $ and   result is clear. Suppose that $ d\geq3 $. As $ n\geq d-t+1=(d-2)-(t-2)+1 $, if $ t=2 $ then  $ \delta^{n}_{\bar{\bar{R}}}(\bar{\bar{R}}/(\bar{\bar{I}})^{m})=0 $, therefore (\ref{4.1}) implies that
       \[\begin{array}{rl}
            \delta_R^{n}(R/I^{m})
            &\leq \delta^{n}_{\bar{R}}(\bar{R}/(\bar{I})^{m})+\delta_{\bar{R}}^{n-1}(\oplus R/I)) \\
            &\leq\delta^{n}_{\bar{\bar{R}}}(\bar{\bar{R}}/(\bar{\bar{I}})^{m})+\delta_{R/(x,y)}^{n-1}(\oplus R/I)+\delta_{R/xR}^{n-1}(\oplus R/I)\\&
            =0.
          \end{array}\]
          
      For the case $ t\geq3 $, we proceed by the same argument as above to find  $ \delta_R^{n}(R/I^{m})=0$. 
    \end{proof}

    \begin{cor}\label{cor}
       Suppose that $(R, \fm)$ is a  Gorenstein local ring  of dimension $d$ such that $R/\fm $ is infinite. Consider the following statements.
       \begin{itemize}
            	\item [(a)]  $R$ is not regular.
             \item [(b)] There exists an  $ \fm $-primary ideal $I$ of $R$ such that
         \begin{itemize}
        \item[(i)]  $ I^{i}/I^{i+1} $ is free $R/I $-module for any $ i\geq0 $, and
        \item[(ii)]  for any $R$-regular sequence ${\bf x}=x_1,\cdots, x_s$ in $ I $ such that $$ x_i+(x_1,\cdots, x_{i-1}) \in ( I/(x_1,\cdots, x_{i-1}))\setminus (I/(x_1,\cdots, x_{i-1}))^{2},\   1\leq i\leq s,$$  one has $\delta^{n}_{R/{\bf x}R}\left( R/I\right)=0$ for all $ n\geq0 $. 
       \end{itemize}
  \item [(c)] There exists a non-zero ideal $I$ of $R$ such that
  \begin{itemize}                       
  \item[(i)] $ \emph\depth( G )\geq \emph\depth_R(R/I) $, and
 \item[(ii)] $\delta_R^{n}(R/I^{m})=0 $ for all  integers $  n\geq d-\emph\depth( G )+1 $ and $ m\geq1 $. 
  \end{itemize}
\end{itemize}
 Then  the implications (a)$\Rightarrow$(b)  and  (b)$\Rightarrow$(c)  hold true. If $ \emph\depth(G)>\emph\depth_R(R/I) $, the statements (a), (b), and (c) are equivalent. 
       \end{cor}
       \begin{proof}
        (a)$\Rightarrow$(b). We show that $ I=\fm $ works. Assume that  ${\bf x}=x_1,\cdots, x_s$ is $R$-regular sequence in $ \fm $ such that  $ x_1\in \fm \setminus \fm^{2}$ and $ x_i+(x_1,\cdots, x_{i-1}) \in ( \fm/(x_1,\cdots, x_{i-1}))\setminus (\fm/(x_1,\cdots, x_{i-1}))^{2}$. Set $ \bar{R}=R/(x_1,\cdots, x_{s-1})R $,  $ \bar{\fm}=\fm/(x_1,\cdots, x_{s-1})R $ and $\bar x_s= x_s+(x_1,\cdots, x_{s-1}) $. As $ \bar R/\bar{x_s}\bar R $ is not regular so by  Remark \ref{12} we have, for all $n\geq 0$, 
        \[\begin{array}{rl}
          \delta^{n}_{ R/ (  x_1,\cdots, x_s) R}\left(  R/ \fm \right)
         & =\delta^{n}_{(\bar R/ \bar{x_s}\bar R) }\left(  R/ \fm \right)
          \\&=\delta^{n}_{(\bar R/ \bar{x_s}\bar R) }\left(  \bar R/ \bar {\fm} \right)\\&
          =\delta^{n}_{(\bar R/ \bar{x_s}\bar R) }\left(  (\bar R/\bar{x_s}\bar R)/ (\bar \fm/\bar{x_s}\bar R) \right)\\&
          =0.    
        \end{array}\]

       (b)$\Rightarrow$(c).  Apply Theorem \ref{sf}.
       
        (c)$\Rightarrow$(a).   We assume that $ R $ is a regular ring. Our assumption, Observation \ref{n} and  Auslander-Buchsbaum formula implies that $ d-\depth_R(R/I)=\pd_R(R/I)\leq d-\depth( G ) $ which contradicts  that $ \depth(G)>\depth_R(R/I) $.
       \end{proof}
       	\section{characterization by linkage theory}
       	The notion of linkage of ideals in commutative algebra is invented by Peskine and Szpiro \cite{PS}.  Two ideals $I$ and $J$ in a Cohen-Macaulay local ring $R$ are said to be linked if there is a
       	regular sequence $\underline{a}$ in their intersection such that $I = (\underline{a}) :_R J$ and $J = (\underline{a}) :_R I$. They have shown that the Cohen-Macaulay-ness property
       	is preserved under linkage over Gorenstein local rings and provided a counterexample to show that
       	the above result is no longer true if the base ring is Cohen-Macaulay but not Gorenstein. In the following, we investigate the situation over a Cohen-Macaulay local ring with canonical module and generalize the result of Peskine and Szpiro \cite{PS}.
       
       \begin{thm} \label{H}
      Let $(R, \fm)$ be a Cohen-Macaulay local ring of dimension $ d $ with canonical module $\omega_R$. Suppose that  $ I $ and   $ J $ are two ideals of $ R $ such that $ 0:_{\omega_R}I=J \omega_R$, $ 0:_{\omega_R}J=I \omega_R$,  $ \emph\gd_{R/I}(\omega_{R}/I\omega_R)<\infty $ and also $\emph\gd_{R/J}(\omega_{R}/J\omega_R)<\infty $ (e.g. $R$ is Gorenstein), then $R/I $ is Cohen-Macaulay $ R $-module  if and only if $R/J $ is  Cohen-Macaulay $ R $-module.
       \end{thm} 
       \begin{proof}
        Assume that  $R/I $ is Cohen-Macaulay. Set $ t:=\gr(I, R)$ so that $t=\h_R(I)=\dim R- \dim R/I $.  If $ t>0  $ then there exists an $R$-regular element $x$ in $ I $. As $ \omega_R $  is maximal Cohen-Macaulay, $ x $ is also $ \omega_R $-regular which implies that $J\omega_R=(0:_{\omega_R}I)=0$. Hence $J=0$ which is absurd. So assume that    $ t=0 $ which implies that $R/I $ is maximal Cohen-Macaulay $R$--module so that $\Ext^{i}_R(R/I, \omega_R)=0  $ for all $ i\geq 1 $. Apply the functor  $ \Hom_R(-, \omega_R) $ on a minimal free resolution  $$ \cdots \longrightarrow \overset{t_2}{\oplus}R\longrightarrow \overset{t_1}{\oplus}R\longrightarrow R\longrightarrow R/I\longrightarrow0 $$ of  $R/I $, to obtain the induced exact sequence 
              	
              	\begin{center}
              		$0\longrightarrow\omega_{R}/J\omega_R\longrightarrow\overset{t_1}{\oplus}\omega_R\longrightarrow \overset{t_2}{\oplus}\omega_R\longrightarrow\cdots.$
              	\end{center}
       Splitting into the short exact sequences 
              	\begin{center}
              		
              		$\begin{array}{ccccccccc}
              		0 & \longrightarrow & \omega_{R}/J\omega_R &  \longrightarrow& \overset{t_1}{\oplus }\omega_R & \longrightarrow & C_1 &\longrightarrow&0  \\ 
              		0& \longrightarrow &C_1  &\longrightarrow  & \overset{t_2}{\oplus }\omega_R & \longrightarrow & C_2 &\longrightarrow  &0  \\ 
              		0 & \longrightarrow & C_2 &\longrightarrow  &\overset{t_3}{\oplus }\omega_R  &\longrightarrow  &C_3  &\longrightarrow  &0 \\ 
              		&  &  & &\vdots, &  &  &  & 
              		\end{array} $
              	\end{center}
              	where $ C_{i}=\im{f_{i+1}}$ for $ i\geq1 $,  we obtain $\depth_R(\omega_R/J\omega_R)=d$. Note that $\gd_{R/J}(\omega_{R}/J\omega_R)<\infty $, implies that $d= \depth_{R/J}(\omega_{R}/J\omega_R)\leq\depth_{R/J}(R/J)$.  Thus  
              	$R/J $ is also a maximal Cohen-Macaulay $R$-module.
       	
       \end{proof}

      To see some applications of Theorem \ref{H}, we refer to the  $ n$th $ \delta $-invariant of an $R$--module $M$ as in the paragraph just after \cite[Proposition 5.3]{AB}.
       \begin{cor}\label{Z}
       	Let $(R, \fm)$ be a Cohen-Macaulay local ring of dimension $ d $ with canonical module   $\omega_R$. Let $ I $ and $ J $ be ideals of $R$.
       	\begin{itemize}
       		\item [(a)] If $ 0:_{\omega_R}I=J \omega_R$ and $R/I $ is a  maximal Cohen-Macaulay $R$-module, then $\delta_{R}^{i}(J\omega_R)=0$ for all $ i\geq1 $.
       		\item[(b)]  If $ 0:_{\omega_R}I=J \omega_R$,  $ 0:_{\omega_R}J=I \omega_R$, $R/I $ is a maximal Cohen-Macaulay $R$-module, and $\emph\gd_{R/J}(\omega_{R}/J\omega_R)<\infty $,  then $\delta_{R}^{i}(I\omega_R)=0$ for all $ i\geq1 $.
       	\end{itemize} 
       \end{cor}
       \begin{proof}
       	(a).  A similar argument as in the proof of Theorem \ref{H}, implies that  $ \depth_R(\omega_{R}/J\omega_R)=d $ and $ \omega_{R}/J\omega_R$ is maximal Cohen-Macaulay.  By  the paragraph just after  \cite[Proposition 5.3]{AB}, we get $\delta_{R}^{i}(J\omega_R)=0$ for all $ i\geq1 $.
       	
       	(b). By Theorem \ref{H}, $R/J $ is maximal Cohen-Macaulay $R$-module so, by  part (a), $\delta_{R}^{i}(I\omega_R)=0$ for all $ i\geq1 $.
       \end{proof}
       	\begin{dfn}
       		\emph{Let $(R, \fm)$ be a local ring. An ideal of $R$ is said to be a linked ideal if $I=0:_R(0:_RI)$.  We call an ideal $I$ to be {\it generically linked} if for any $\fp\in\Spec R$ with $\h\fp=0$, the ideal $IR_\fp$ is linked in $R_\fp$.}
       	\end{dfn}
       	Note that a local ring $R$ is generically Gorenstein if and only if any ideal of height $0$ is generically linked.
       	
       In the following, we recall a characterization for a Cohen-Macaulay ring to be Gorenstein in terms of linkage of particular ideals. The proof is straightforward and we bring it here for convenience of the readers.  
       		\begin{prop}\label{gor}
       			Let $(R, \fm)$ be a Cohen-Macaulay local ring. The following statements are equivalent.
       			\begin{itemize}
       				\item [(i)] $R$ is Gorenstein.
       				\item[(ii)] For any non-zero proper ideal $I$ of $R$ and each $R$-regular sequence ${\bf x}=x_1,\cdots, x_t$ in $I$ with $t=\emph{\h} I$, the ideal $I/({\bf x})$ is generically linked in $R/({\bf x})$.
       			\end{itemize}
       	\end{prop}
       	\begin{proof} 
       			(i)$\Rightarrow$(ii). Let that $I$ is an ideal of $R$, $\h I=t$. Let ${\bf x}=x_1,\cdots, x_t\in I$ be an $R$-sequence. Then $\h (I/({\bf x}))=0$ in $R/({\bf x})$. As $R/({\bf x})$ is Gorenstein, $I/({\bf x})$ is generically linked in $R/({\bf x})$.

       	 (ii)$\Rightarrow$(i). We prove by induction on $d:=\dim R$. In the case $d=0$, we have $I=0:_R(0:_RI)$ for all ideals $I$ by our assumption. Therefore $R$ is Gorenstein. Let $d>0$ and the result has been settled for rings of dimensions smaller than $d$. Choose an $R$--regular element $x$ and set $\overline{R}=R/Rx$. We note that $\overline{R}$ is a Cohen-Macaulay local ring of dimension $d-1$. Let  $\overline{I}:=I/Rx$ be an ideal of $\overline{R}$ with height $t$  so that there is an $\overline{R}$--regular sequence $\overline{y}_1, \cdots, \overline{y}_t$, with $y_1, \ldots, y_t\in I$. Note that ${\bf x}:= x, y_1, \cdots, y_{t}$ is an $R$--sequence contained in $I$ and $\h I=t+1$. By our assumption, $I/(x, y_1, \cdots, y_t)$ is generically linked in $R/(x, y_1, \cdots, y_{t})$. Equivalently, the ideal $\frac{I/Rx}{(\overline{y_1}, \cdots, \overline{y_t})}$ is generically linked in $R/Rx$. By induction hypothesis, $R/Rx$ is Gorenstein and so is $R$.
       	\end{proof}

         	We end the paper by the following remark which shows that generically Gorenstein local rings possessing canonical modules with dimensions greater than $1$ are reducible to the one with lower dimension by some non-zero divisor.
              	\begin{rmk}\label{rg}
              		Let $(R, \fm)$ be a Cohen-Macaulay local ring with canonical module $\omega_{R}$ and $\emph{\dim} R>1$. Then the following are equivalent. 
              			\begin{itemize}
              				\item [(i)] $R$ is generically Gorenstein.
              				\item [(ii)] There exists an $R$--regular element $x$ such that $R/Rx$ is generically Gorenstein.
              		\end{itemize}
              	\end{rmk}  
              	\begin{proof}
              			(i)$\Rightarrow$(ii). We may assume that $R$ is not Gorenstein. As $R$ is generically Gorenstein, by \cite[Proposition 3.3.18]{BH}, $\omega_{R}$ is an ideal of $R$. As $\fm\not\subseteq(\underset{\fp\in\Ass R}{\cup }\fp)\cup(\underset{\fq\in\Ass_R( R/\omega_R)}{\cup}\fq)$, there is $x\in \fm$ which is regular on $R$, $R/\omega_R$. We define  $R$-homomorphism $\alpha :\omega_R/x\omega_R\longrightarrow R/xR$ such that $ \alpha(u+x\omega_R)=u+xR $. As $\alpha$ is injective, $\omega_{R}/x\omega_{R}$ can be embedded into $R/Rx$. As $\omega_{R}/x\omega_{R}\cong\omega_{R/Rx}$ is the canonical module of $R/Rx$, the result follows.  
              		
              (ii)$\Rightarrow$(i). Assume that $R/Rx$ is generically Gorenstein for some $R$-regular element $x$. Let $\fp$ be a minimal prime ideal of $R$. Thus one has $\dim R/(Rx+\fp)=\dim R-1$ and $\h(Rx+\fp)=1$. Hence $Rx+\fp\subseteq\fq$ for some prime ideal $\fq$ with $\h\fq=1$. By our assumption, $(R/Rx)_{\fq/Rx}$  and so $R_\fq/xR_\fq$ is Gorenstein. Hence $R_\fq$ and therefore $R_\fp$ is Gorenstein. 
              
              	\end{proof}
              \begin{eg}
              	\emph{Assume that $R$ is a non-Gorenstein Cohen-Macaulay local ring with canonical module $\omega_R$ of dimension $1$ (e.g.   $R= k \llbracket t^3, t^5, t^7 \rrbracket       $, $k$ is a field). Then, for any $R$-regular element $x$, $R/Rx$ is not generically Gorenstein.}
              \end{eg}

\bibliographystyle{amsplain}

\end{document}